\documentclass[11pt,a4paper]{article}
\usepackage[a4paper,margin=2.3cm]{geometry}
\usepackage{amsmath,amssymb,amsfonts,mathrsfs,amsthm}
\usepackage{array,booktabs,longtable,enumitem}
\usepackage{graphicx,xcolor,hyperref}
\usepackage[T1]{fontenc}
\allowdisplaybreaks
\hypersetup{colorlinks=true,linkcolor=blue,urlcolor=blue,citecolor=blue}
\setlist{nosep}
\newif\ifmarked
\ifdefined\MARKED\markedtrue\else\markedfalse\fi
\ifmarked

\newcommand{\deleted}[1]{\textcolor{red}{#1}}
\else

\newcommand{\deleted}[1]{}
\fi

\newtheorem{theorem}{Theorem}[section]

\newtheorem{lemma}{Lemma}[section]
\newtheorem{cor}[theorem]{Corollary}

\begin{document}
\title
{\LARGE \textbf{Almost complete graphs determined by Laplacian hook immanantal polynomials}}

\author{Shuaijun Li$^{a}$,  Guangfu Wang$^{b,c}$ \thanks{{Corresponding author.\newline
\emph{E-mail address}: gfwang@ytu.edu.cn
}},\\
{\small $^{a}$ School of Mathematics and Statistics, Qinghai Minzu University, }\\
{\small  Xining, Qinghai 810007, P.R.~China} \\
{\small $^{b}$ School of Mathematical and Information Sciences, Yantai University,}\\
{\small Yantai, Shandong 264000, P.R.~China}\\
{\small $^{c}$ ECOPRO, Institute for Basic Science,}\\
{\small 55 Expo-ro, Yuseong-gu, Daejeon, 34126, Korea}
}

\date{}

\maketitle

\noindent {\bf Abstract: }
Let \(\mathscr{G}_n\) be the family of simple graphs obtained from \(K_n\) by deleting at most five edges. For a fixed integer \(1\leq k\leq n\), let \(\Phi_k(L(G),x)\) denote the immanantal polynomial of the Laplacian matrix associated with the hook partition \((k,1^{n-k})\). We prove that, for \(n>7\) and \(n\neq 2k-1\), every graph in \(\mathscr{G}_n\) is determined by \(\Phi_k(L(G),x)\) among all simple graphs.  We prove that, for \(n>7\) and \(n\neq 2k-1\), every graph in \(\mathscr{G}_n\) is determined by \(\Phi_k(L(G),x)\) among all simple graphs. The proof recovers the order, size, and degree-square sum from the first coefficients, and then separates the remaining candidates by explicit third- and fourth-coefficient comparisons based on the finite classification of complements with at most five edges. The case \(n=2k-1\) is left open because the binomial differences used in these comparisons vanish.

\smallskip


\section{Introduction}
Let \(G=(V(G),E(G))\) be a simple graph with vertex set \(V(G)=\{v_1,v_2,\dots,v_n\}\) and edge set \(E(G)\). Let \(d_{v_i}\) denote the degree of vertex \(v_i\), and let \(D(G)=\mathrm{diag}(d_{v_1},d_{v_2},\dots,d_{v_n})\) be the degree diagonal matrix of G. The adjacency matrix of G is an \(n\times n\) matrix denoted by \(A(G)=[a_{ij}]\), where \(a_{ij}=1\) if vertices \(v_i\) and \(v_j\) are adjacent, and \(a_{ij}=0\) otherwise. The Laplacian matrix of graph G is defined as
\(L(G)=D(G)-A(G).\)

Let $n$ be a positive integer, and let \(S_n\) denote the symmetric group of order $n$. For any partition \(\lambda = (\lambda_1,\lambda_2,\dots,\lambda_n)\), let \(\chi_\lambda\) stand for the irreducible character of \(S_n\). For any \(n\times n\) matrix M, the immanant function of M associated with \(\chi_\lambda\) is defined as
$$d_\lambda(M) = \sum_{\sigma\in S_n} \chi_\lambda(\sigma) \prod_{i=1}^n m_{i,\sigma(i)}.$$
When \(\lambda = (k,1^{n-k})\), \(d_{(k,1^{n-k})}(M)\) is called the hook immanant of M. In particular, the determinant \(\det M\) and permanent \(\operatorname{per} M\) correspond respectively to the alternating character \(\chi_{(1^n)}\) and the trivial character \(\chi_{(n)}\).  B\"{u}rgisser \cite{Bu1} proved that the computation of hook immanants is both \(\#\text{P}\)-complete and VNP-complete as the polynomial width grows. As a special graph invariant, hook immanants have profound intrinsic connections with the structural properties of graphs and have been widely studied in algebraic combinatorics~\cite{CD,CF}. Morris \cite{Merris1} uncovered the inherent relation between the hook immanant of the adjacency matrix and the number of Hamiltonian cycles in a graph. Chan and Lam \cite{Chan1} characterized the minimum bound of the immanants of Laplacian matrices of trees.The theories concerning matrix immanants and the immanantal properties of graph matrices have been systematically elaborated in numerous references; see~\cite{D,I}.

Let \(I_n\) denote the \(n\times n\) identity matrix. The hook immanantal polynomial \(\Phi_k(M,x)\) of a matrix M is defined as \(d_{(k,1^{n-k})}(xI_n - M)\), i.e.,
$$\Phi_k(M,x) = d_{(k,1^{n-k})}(xI_n - M).$$
In particular, suppose M is a graph matrix. When \(k=1\), \(\Phi_1(M,x)\) is called the characteristic polynomial of the graph matrix M; when \(k=2\), \(\Phi_2(M,x)\) is referred to as the second immanantal polynomial of M; when \(k=n\), \(\Phi_n(M,x)\) is termed the permanent polynomial of $M$. The polynomial \(\Phi_k(L(G),x) = d_{(k,1^{n-k})}(xI_n - L(G))\) is named the $k$-th hook immanantal polynomial of graph $G$. The cases \(k=1,2,n\) recover the characteristic, second immanantal, and permanent polynomials, respectively. The characteristic polynomials and permanent polynomials of graph matrices have been extensively investigated \cite{Liu1}. Morris~\cite{Merris3} pointed out that almost all trees possess a complete set of Laplacian immanantal polynomials. Cash \cite{Cash1} generalized the Sachs theorem for adjacency matrices to immanantal polynomials and studied their structural properties of chemical molecular graphs. Yu and Qu~\cite{Yu1}~derived an explicit expression for the immanantal polynomial of a graph matrix M via fundamental subgraphs. In algebraic combinatorics, immanants corresponding to several special partitions have been widely researched \cite{DM,Wu3,Wu4}. For further investigations on immanantal polynomials, readers may refer to~\cite{AE,Nagar1,Merris4}.

For a fixed \(k\), two n-vertex graphs \(G\) and \(H\) are called \(k\)-th hook-immanantally cospectral if \(\Phi_k(L(G),x)=\Phi_k(L(H),x)\). If they are non-isomorphic, they form a pair of \(k\)-th hook-immanantal partners. If equality holds for every \(1\leq k\leq n\), we say that \(G\) and \(H\) have the same complete hook-immanantal polynomial family. A graph $G$ is said to be determined by its $k$-th hook immanantal polynomial if $G$ has no $k$-th hook immanantal partner.

A classical open problem in spectral graph theory reads: which graphs are determined by their polynomials? This problem is notoriously difficult within the theory of graph polynomials and has attracted substantial attention from researchers in recent years. Van Dam and Haemers \cite{ER,van1} systematically investigated polynomials associated with various graph matrices and remarked that characterizing graphs which are determined by a given graph polynomial constitutes an intriguing research direction. Morris et al. \cite{Merris5} first initiated the study of the question: which graphs are determined by their Laplacian permanent polynomials? Liu \cite{Liu2} proved that both the complete graph \(K_n\) and the star graph \(S_n\) are uniquely determined by their Laplacian permanent polynomials. Liu et al. \cite{liux1} verified that a family of unicyclic graphs are uniquely determined by their Laplacian characteristic polynomials. For further relevant results, the reader is referred to [23,24].
\begin{figure}[h]
\centering
\includegraphics[scale=0.6]{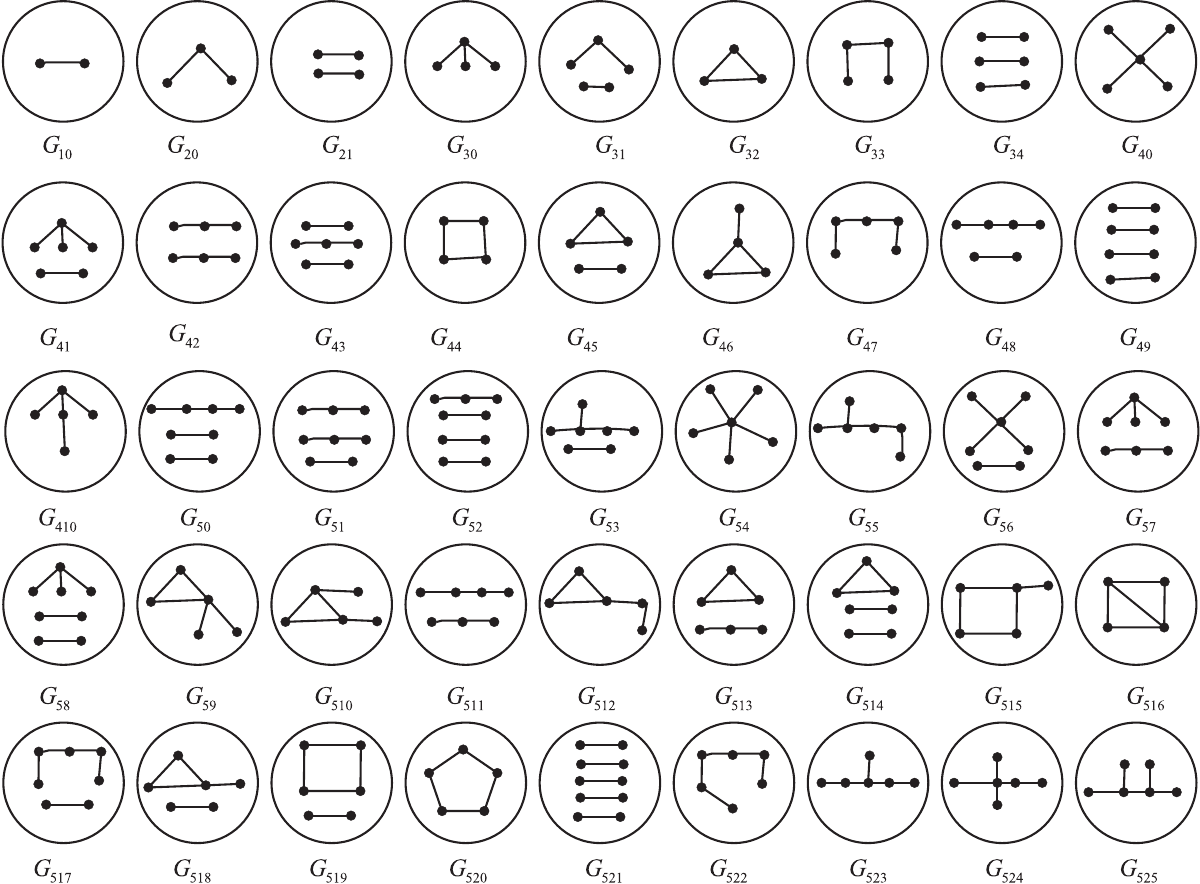}
\caption{$\mathscr{G}_{n}$:All subgraphs obtained by deleting 1 to 5 edges from the complete graph \(K_n\); the displayed labels refer to the deleted-edge complements.}
\end{figure}

Let \(\mathscr{G}_n\) be the family of simple graphs obtained from \(K_n\) by deleting at most five edges. Zhang and Wu characterized a graph family \(\mathscr{G}_n\) in ~\cite{H1}, which consists of all graphs obtained by deleting at most five edges from the complete graph \(K_n\) (see Figure 1). We write \(G_{ij}=K_n-E(F_{ij})\), where \(F_{ij}\) is the small deleted-edge graph represented by the label \(G_{ij}\) in Figure 1. C\'{a}mara and Haemers \cite{ca} proved that every graph in \(\mathscr{G}_n\) is uniquely determined by the characteristic polynomial of its adjacency matrix. Zhang and Wu \cite{H1} further showed that all graphs in \(\mathscr{G}_n\) are uniquely determined by their Laplacian permanent polynomials. Zeng and Wu verified that graphs in \(\mathscr{G}_n\) can be uniquely identified by their second Laplacian immanantal polynomials.Since the characteristic polynomial, permanent polynomial, and second immanantal polynomial are all special cases of hook immanantal polynomials, a natural question arises: are all graphs in \(\mathscr{G}_n\) uniquely determined by their $k$-th hook immanantal polynomials?

We discuss the above question and derive the following theorem:

\begin{theorem}\label{thm:main}
For any \(1\leq k\leq n\) and any $n$-vertex graph \(G\in\mathscr{G}_n\), if \(n>7\) and \(n\neq 2k-1\), then \(G\) is uniquely determined by its \(k\)-th hook immanantal polynomial.
\end{theorem}

The remainder of this paper is organized as follows. In Section 2, we derive several special properties of the Laplacian hook immanantal polynomials and present some inequalities for combinatorial sequences. In Section 3, we provide the proof of Theorem 1.1. In Section 4, we summarize the main conclusions of this paper and propose several open problems.
\section{Preliminaries}

Let $G$ be a simple graph with vertex set \(V(G)=\{v_1,v_2,\dots,v_n\}\). Denote by \(G-e\) the graph obtained by deleting an edge $e$ from $G$.
Let~$\mathfrak{D}(G)$~stand for the degree sequence of G, that is, ~$\mathfrak{D}(G) = (d(v_1), d(v_2), \dots, d(v_n))$. Let \(\mathscr{C}(G)\) and \(\mathscr{M}(G)\) denotedenote the sets consisting of all cycles of length $l$ and all $l$-matchings in $G$, respectively.
Let~$\mathscr{T}_j(G)$~stand for the sum of the degrees of the three vertices of the $j$-th triangle in~$\mathscr{C}_3(G).$~
If~$C = v_{i_1}v_{i_2}\cdots v_{i_l}v_{i_1} \in \mathcal{C}_l(G)$, then let~$\mathfrak{D}_C(G)$~denote the degree sequence obtained by deleting ~$d(v_{i_1}), d(v_{i_2}), \dots, d(v_{i_l})$~from ~$\mathfrak{D}(G),$~
If ~$\mathfrak{M} = \{v_{i_1}v_{i'_1}, v_{i_2}v_{i'_2}, \dots, v_{i_l}v_{i'_l}\} \in \mathscr{M}_l(G)$, then let ~$\mathfrak{D}_{\mathfrak{M}}(G)$~denote the degree sequence obtained by deleting ~$d(v_{i_1}), d(v_{i'_1}), d(v_{i_2}), d(v_{i'_2}), \dots, d(v_{i_l}), d(v_{i'_l})$~from~$\mathfrak{D}(G)$,~
Let \(F_r\) denote the $r$-th elementary symmetric function, and define$F_r(G) = F_r(\mathfrak{D}(G))$. For any integer \(r \ge l\), we define
$$\mathcal{C}_r^l(G) = \sum_{C \in \mathscr{C}_l(G)} F_{r-l}(\mathfrak{D}_C(G)),$$
and
$$\mathcal{M}_r^l(G) = \sum_{\mathfrak{M} \in \mathscr{M}_l(G)} F_{r-2l}(\mathfrak{D}_{\mathfrak{M}}(G)).$$

\begin{lemma}
(\cite{Dong})Let $G$ be a graph with $n$ vertices and $m$ edges, and $(d(v_{1}),d(v_{2}),\ldots,d(v_{n}))$ denote the degree sequence of $G$. Denote
\begin{eqnarray*}
\Phi_{(k,1^{n-k})}(xI -L(G)) = \sum_{r=0}^{n} (-1)^r c_{(k,1^{n-k}),r}(L(G))x^{n-r}.
\end{eqnarray*}
Then
\begin{eqnarray*}
c_{(k,1^{n-k}),0}(L(G)) &=& \binom{n-1}{k-1}, \\
c_{(k,1^{n-k}),1}(L(G)) &=& F_1(G)\binom{n-1}{k-1} = 2m\binom{n-1}{k-1}, \\
c_{(k,1^{n-k}),2}(L(G)) &=& F_2(G)\binom{n-1}{k-1} +m\bigg[\binom{n-2}{k-2}-\binom{n-2}{k-1}\bigg]\quad (n \geq 3), \\
c_{(k,1^{n-k}),3}(L(G)) &=& F_3(G)\binom{n-1}{k-1} +\mathcal{M}_3^{1}(G)\bigg[\binom{n-2}{k-2}-\binom{n-2}{k-1}\bigg]\\
&&-2|\mathscr{C}_{3}(G)|\bigg[\binom{n-4}{k-4}-\binom{n-4}{k-1}\bigg]  \quad (n \geq 4),\\
c_{(k,1^{n-k}),4}(L(G)) &=& F_4(G)\binom{n-1}{k-1} +\mathcal{M}_4^{1}(G)\bigg[\binom{n-2}{k-2}-\binom{n-2}{k-1}\bigg]\\
&&-2\mathcal{C}_{4}^{3}(G)\bigg[\binom{n-4}{k-4}-\binom{n-4}{k-1}\bigg]+2|\mathscr{C}_{4}(G)|\bigg[\binom{n-5}{k-5}-\binom{n-5}{k-1}\bigg] \\
&&+\bigg[\binom{m}{2}-\sum\limits_{i=1}^{n} \binom{d(v_{i})}{2}\bigg]\bigg[\binom{n-5}{k-5}-2\binom{n-5}{k-3}+\binom{n-5}{k-1}\bigg] \quad (n \geq 5),
\end{eqnarray*}
where \(|\mathcal{C}_l(G)|\) denotes the number of cycles of length l in G. Meanwhile, if \(b<0\) or \(b>a\), we stipulate that
\(\binom{a}{b}=0.\)

\end{lemma}
\begin{lemma}
(\cite{Liu2})
Let $G$ be a simple connected graph with $n$ vertices and $m$ edges, and let \((d_1,d_2,\dots,d_n)\) denote its degree sequence, \(S_r(G)=\sum_{i=1}^{n}d_i^r\). Then: 
\begin{align*}
F_2(G) &= -\frac{1}{2}\sum_{i=1}^{n} d_i^2 + 2m^2,\\
F_3(G) &= \frac{1}{3}\sum_{i=1}^{n} d_i^3 - m\sum_{i=1}^{n} d_i^2 + \frac{4}{3}m^3,\\
F_4(G) &= -\frac{1}{4}\sum_{i=1}^{n} d_i^4 + \frac{2}{3}m\sum_{i=1}^{n} d_i^3 - m^2\sum_{i=1}^{n} d_i^2 + \frac{1}{8}\Bigl(\sum_{i=1}^{n} d_i^2\Bigr)^2 + \frac{2}{3}m^4,\\
\mathcal{M}_3^{1}(G) &= 2m^2 - \sum_{i=1}^{n} d_i^2,\\
\mathcal{M}_4^{1}(G) &= \sum_{i=1}^{n} d_i^3 - \frac{5}{2}m\sum_{i=1}^{n} d_i^2 + \sum_{v_i v_j \in E(G)} d_i d_j + 2m^3,\\
 \mathcal{C}_{4}^{3}(G)&= 2m|\mathscr{C}_3(G)|-\sum_{C\in\mathscr{C}_3(G)}\sum_{v\in V(C)}d(v)\\
&=2m|\mathscr{C}_3(G)|-\sum_j\mathscr{T}_j(G).
\end{align*}
\end{lemma}

\begin{lemma}\label{lem:global-reduction}
	Let \(G\in\mathscr{G}_n\), and let \(H\) be any simple graph. For a fixed
	\(1\leq k\leq n\), if
	\(\Phi_k(L(G),x)=\Phi_k(L(H),x)\), then \(H\in\mathscr{G}_n\).
\end{lemma}

\begin{proof}
	The degree of the polynomial is the number of vertices, because
	\(c_{k,0}=\binom{n-1}{k-1}>0\). Thus \(H\) has \(n\) vertices. Equality of
	the leading coefficient and of the coefficient of \(x^{n-1}\), together
	with Lemma~2.1, gives
	\[
	|E(H)|=|E(G)|=\binom{n}{2}-t,\qquad 0\leq t\leq5.
	\]
	Every simple n-vertex graph with this many edges is obtained from \(K_n\)
	by deleting \(t\) edges, so \(H\in\mathscr{G}_n\).
\end{proof}

\begin{lemma}
For any graph $G$, its hook immanantal polynomial can determine the quantities: $n$, $m=|E(G)|$ and \(\displaystyle\sum_{i=1}^{n} d_i^2\). 
\end{lemma}

\begin{proof}
The degree determines \(n\). Put \(c_0=\binom{n-1}{k-1}\). Then
\[
m=\frac{c_{k,1}(G)}{2c_0},
\qquad
\sum_{i=1}^{n}d_i^2
=4m^2-\frac{2}{c_0}\left(
c_{k,2}(G)-m\left[\binom{n-2}{k-2}-\binom{n-2}{k-1}\right]\right).
\]
This formula uses the convention \(\binom{a}{b}=0\) for \(b<0\) or
\(b>a\), and therefore also covers \(k=1\).
\end{proof}

\begin{lemma}
(\cite{wu6})
When \(n>7\), the complete graph with \(1,2,3,4\) edges removed is determined by its Laplacian characteristic polynomial.
\end{lemma}
\begin{lemma}
(\cite{wu5})
When \(n>7\), the complete graph with 5 edges removed is determined by its Laplacian characteristic polynomial.
\end{lemma}
\begin{lemma}
(\cite{zeng1})
All graphs in \(\mathscr{G}_n\) are determined by their second immanantal polynomial.
\end{lemma}

Wu and Zhou \cite{wu2}  computed the sum of squares of degrees for graphs in \(\mathcal{G}_n\) (see Table 1). For the convenience of subsequent proofs, we further calculated the sum of cubes of degrees, the sum of the products of the degrees of vertices on every cycle of length 3, and the sum of the products of the degrees of the two endpoints over all edges for several graphs in \(\mathcal{G}_n\), as listed in Table 2 and Table 3 respectively.
\begin{table}[htbp]
  \centering
  \setlength{\tabcolsep}{2pt}
  \caption{\(\sum\limits_{i=1}^{n} d_i^2\):The sum of squares of degrees  in \(\mathscr{G}_n\)  (Source: \cite{wu2})}
  \begin{tabular}{|c|c|c|c|}
\hline
Graph  & $\sum_{i=1}\limits^n(d_i)^2$ & Graph  & $\sum_{i=1}\limits^n(d_i)^2$ \\
\hline
$G_{20}$ & $n^3 - 2n^2 - 7n + 14$ &$G_{30}, G_{32}$ & $n^3 - 2n^2 - 11n + 24$ \\
 \hline
$G_{33}$ & $n^3 - 2n^2 - 11n + 22$ &$G_{40}$ & $n^3 - 2n^2 - 15n + 36$ \\
 \hline
$G_{42}, G_{48}$ & $n^3 - 2n^2 - 15n + 28$ &$G_{44}, G_{410}$ & $n^3 - 2n^2 - 15n + 32$ \\
 \hline
$G_{49}$ & $n^3 - 2n^2 - 15n + 24$ &$G_{52}$ & $n^3 - 2n^2 - 19n + 32$ \\
 \hline
$G_{54}$ & $n^3 - 2n^2 - 19n + 50$&$G_{56}, G_{512}, G_{515}, G_{525}$ & $n^3 - 2n^2 - 19n + 42$ \\
 \hline
$G_{58}, G_{511}, G_{514}, G_{517}$ & $n^3 - 2n^2 - 19n + 36$ &$G_{510}, G_{524}$ & $n^3 - 2n^2 - 19n + 44$ \\
 \hline
$G_{21}$ & $n^3 - 2n^2 - 7n + 12$ &$G_{31}$ & $n^3 - 2n^2 - 11n + 20$ \\
 \hline
$G_{34}$ & $n^3 - 2n^2 - 11n + 18$ &$G_{41}, G_{45}, G_{47}$ & $n^3 - 2n^2 - 15n + 30$ \\
 \hline
$G_{43}$ & $n^3 - 2n^2 - 15n + 26$ &$G_{46}$ & $n^3 - 2n^2 - 15n + 34$ \\
 \hline
$G_{50}, G_{51}$ & $n^3 - 2n^2 - 19n + 34$ &$G_{59}, G_{516}$ & $n^3 - 2n^2 - 19n + 46$ \\
 \hline
$G_{55}, G_{518}, G_{520}, G_{523}$ & $n^3 - 2n^2 - 19n + 40$ &$G_{521}$ & $n^3 - 2n^2 - 19n + 30$ \\
 \hline
 $G_{53}, G_{57}, G_{513}, G_{519}, G_{522}$ & $n^3 - 2n^2 - 19n + 38$ &$  $&$  $\\
\hline
\end{tabular}
\end{table}
\begin{lemma}
(\cite{Bru1})
Let n be a positive integer. Then the binomial coefficient sequence
$$\binom{n}{0},\,\binom{n}{1},\,\binom{n}{2},\,\dots,\,\binom{n}{n}$$
is unimodal.
\end{lemma}
\begin{lemma}
Let \(n,k\) be two positive integers with \(k\ge 5\). Then
$$
\begin{cases}
\binom{n-4}{k-4} - \binom{n-4}{k-1}<0,\binom{n-5}{k-5} - \binom{n-5}{k-1}<0&n\geq2k, \\
\binom{n-4}{k-4} - \binom{n-4}{k-1}=0,\binom{n-5}{k-5} - \binom{n-5}{k-1}=0&n\doteq2k-1, \\
\binom{n-4}{k-4} - \binom{n-4}{k-1}>0,\binom{n-5}{k-5} - \binom{n-5}{k-1}>0 &n\leq2k-2.
\end{cases}
$$

\end{lemma}
\begin{proof}
Let
\(A = \binom{n-4}{k-4} - \binom{n-4}{k-1},\quad
B = \binom{n-5}{k-5} - \binom{n-5}{k-1}.\)
The case \(n=2k-1\) is straightforward to verify. By Lemma 2.7, the binomial coefficient sequence is unimodal.
If \(k-1 \le \dfrac{n-4}{2}\), i.e., \(n \ge 2k+2\), then \(A < 0\).
If \(k-4 \ge \dfrac{n-4}{2}\), i.e., \(n \le 2k-4\), then \(A > 0\).
When \(n=2k-3\):
\(A = \binom{2k-7}{k-4} - \binom{2k-7}{k-1} > 0.\)
When \(n=2k-2\):
\(A = \binom{2k-4}{k-4} - \binom{2k-4}{k-1} > 0.\)
When \(n=2k\):
\(A = \binom{2k-4}{k-4} - \binom{2k-4}{k-1} < 0.\)
When \(n=2k+1\):
\(A = \binom{2k-3}{k-4} - \binom{2k-3}{k-1} < 0.\)To summarize:
\(A < 0\) whenever \(n \ge 2k\), and \(A > 0\) whenever \(n \le 2k-2\).Similarly, we deduce:
\(B < 0\) for \(n \ge 2k\),
\(B = 0\) for \(n = 2k-1\),
\(B > 0\) for \(n \le 2k-2\).
\end{proof}

\begin{cor} For any positive integers \(a,b,c,n,k\) with \(n\ge k\), \(k\ge 5\) and \(n\neq 2k-1\), we have
$$
\frac{a(2k-n-1)}{k-1}\binom{n-2}{k-2}+b[\binom{n-4}{k-4}-\binom{n-4}{k-1}]+c[\binom{n-5}{k-5}-\binom{n-5}{k-1}]
\neq 0.
$$
\end{cor}

\begin{proof}
If \(n\ge 2k\), then
$\frac{2k-n-1}{k-1}<0, \binom{n-4}{k-4}-\binom{n-4}{k-1}<0, \binom{n-5}{k-5}-\binom{n-5}{k-1}<0,$
so
$\frac{a(2k-n-1)}{k-1}\binom{n-2}{k-2}+b[\binom{n-4}{k-4}-\binom{n-4}{k-1}]+c[\binom{n-5}{k-5}-\binom{n-5}{k-1}]
< 0.$
If \(n\le 2k-2\), then
$\frac{2k-n-1}{k-1}>0, \binom{n-4}{k-4}-\binom{n-4}{k-1}>0, \binom{n-5}{k-5}-\binom{n-5}{k-1}>0,$
so
$\frac{a(2k-n-1)}{k-1}\binom{n-2}{k-2}+b[\binom{n-4}{k-4}-\binom{n-4}{k-1}]+c[\binom{n-5}{k-5}-\binom{n-5}{k-1}]$
Therefore,
$\frac{a(2k-n-1)}{k-1}\binom{n-2}{k-2}+b[\binom{n-4}{k-4}-\binom{n-4}{k-1}]+c[\binom{n-5}{k-5}-\binom{n-5}{k-1}]\neq 0.$
\end{proof}

\begin{table}[htbp]
  \centering
  \setlength{\tabcolsep}{2pt}
   \caption{$\sum_{i=1}\limits^n(d_i)^3$:The sum of cubes of degrees in \(\mathscr{G}_n\)}
  \begin{tabular}{|c|c|c|c|}
    \hline
    Graph & $\sum_{i=1}\limits^n(d_i)^3$ & Graph  & $\sum_{i=1}\limits^n(d_i)^3$ \\
    \hline
    $G_{30}$ & $n^4-3n^3-15n^2+71n-84$ & $G_{511},G_{514},G_{517}$ & $n^4-3n^3-27n^2+107n-106$ \\
    \hline
    $G_{32}$ & $n^4-3n^3-15n^2+71n-78$ & $G_{512},G_{515}$ & $n^4-3n^3-27n^2+125n-148$ \\
    \hline
    $G_{41}$ & $n^4-3n^3-21n^2+89n-98$ & $G_{513},G_{519},G_{522}$ & $n^4-3n^3-27n^2+113n-118$ \\
    \hline
    $G_{42},G_{48}$ & $n^4-3n^3-21n^2+83n-80$ & $G_{516}$ & $n^4-3n^3-27n^2+137n-178$ \\
    \hline
    $G_{44}$ & $n^4-3n^3-21n^2+95n-104$ & $G_{518},G_{523},G_{55}$ & $n^4-3n^3-27n^2+119n-136$ \\
    \hline
    $G_{45},G_{47}$ & $n^4-3n^3-21n^2+89n-92$ & $G_{520}$ & $n^4-3n^3-27n^2+119n-130$ \\
    \hline
    $G_{410}$ & $n^4-3n^3-21n^2+95n-110$ & $G_{524}$ & $n^4-3n^3-27n^2+131n-178$ \\
    \hline
    $G_{50},G_{51}$ & $n^4-3n^3-27n^2+101n-94$ & $G_{525}$ & $n^4-3n^3-27n^2+125n-154$ \\
    \hline
    $G_{53},G_{57}$ & $n^4-3n^3-27n^2+113n-124$ & $G_{59}$ & $n^4-3n^3-27n^2+137n-190$ \\
    \hline
    $G_{56}$ & $n^4-3n^3-27n^2+125n-166$ & $G_{510}$ & $n^4-3n^3-27n^2+131n-166$ \\
    \hline
    $G_{58}$ & $n^4-3n^3-27n^2+107n-112$ & $  $&$  $\\
    \hline
  \end{tabular}
\end{table}

\begin{lemma}
Let \(n,k\) be two positive integers with \(n \ge k\) and \(n \ne k-1\). Then
$$\binom{n-1}{k-1} \ge \binom{n-4}{k-4}.$$
\end{lemma}
\begin{proof}
The cases \(k=1,2,3\) are trivial to verify.

For \(k \ge 4\), direct computation yields:
\[
\binom{n-1}{k-1} = \frac{(n-1)!}{(k-1)!(n-k)!}, \quad \binom{n-4}{k-4} = \frac{(n-4)!}{(k-4)!(n-k)!}
\]
Thus
\[
\frac{\binom{n-1}{k-1}}{\binom{n-4}{k-4}} = \frac{(n-1)(n-2)(n-3) \cdot (n-4)!}{(k-1)(k-2)(k-3) \cdot (k-4)!} \cdot \frac{(k-4)!}{(n-4)!} = \frac{(n-1)(n-2)(n-3)}{(k-1)(k-2)(k-3)}
\]
Since \(n \ge k \ge 4\) and \(n,k\) are positive integers, we have
\[
n-1 \geq k-1 > 0, \quad n-2 \geq k-2 > 0, \quad n-3 \geq k-3 > 0
\]
Hence
$\displaystyle \frac{(n-1)(n-2)(n-3)}{(k-1)(k-2)(k-3)} \geq 1$.

To sum up,
~$\displaystyle \binom{n-1}{k-1} \geq \binom{n-4}{k-4}~.$
\end{proof}

\begin{table}[htbp]
  \centering
  \setlength{\tabcolsep}{2pt}
  \caption{The sum of the products of the degrees of all vertices of cycles of length 3 and the sum of the products of the degrees of all edges in \(\mathscr{G}_n\)}.
  \begin{tabular}{|l|c|c|}
    \toprule
    Graph & $ \sum_i\limits \mathscr{T}_j(G)$& $\sum_{v_iv_j\in E(G)}\limits d_{v_i}d_{v_j}$ \\
    \midrule
    $G_{42}$ & $(1/2)n^4 - 2n^3 - (27/2)n^2 + 65n - 50$ & $(1/2)n^4 - (3/2)n^3 - (21/2)n^2 + (71/2)n - 6$ \\
    \hline
    $G_{50}$ & $(1/2)n^4 - 2n^3 - (35/2)n^2 + 79n - 48$ & $(1/2)n^4 - (3/2)n^3 - (27/2)n^2 + (87/2)n + 4$ \\
    \hline
    $G_{53}$ & $(1/2)n^4 - 2n^3 - (35/2)n^2 + 89n - 79$ & $(1/2)n^4 - (3/2)n^3 - (27/2)n^2 + (95/2)n - 7$ \\
    \hline
    $G_{55}$ & $(1/2)n^4 - 2n^3 - (35/2)n^2 + 94n - 94$ & $(1/2)n^4 - (3/2)n^3 - (27/2)n^2 + (99/2)n - 13$ \\
    \hline
    $G_{57}$ & $(1/2)n^4 - 2n^3 - (35/2)n^2 + 89n - 75$& $(1/2)n^4 - (3/2)n^3 - (27/2)n^2 + (95/2)n - 5$ \\
    \hline
    $G_{519}$ & $(1/2)n^4 - 2n^3 - (35/2)n^2 + 89n - 80$ & $(1/2)n^4 - (3/2)n^3 - (27/2)n^2 + (95/2)n - 9$ \\
    \hline
    $G_{511}$ & $(1/2)n^4 - 2n^3 - (35/2)n^2 + 84n - 61$ & $(1/2)n^4 - (3/2)n^3 - (27/2)n^2 + (91/2)n - 1$ \\
    \hline
    $G_{48}$ & $(1/2)n^4 - 2n^3 - (27/2)n^2 + 65n - 52$ & $(1/2)n^4 - (3/2)n^3 - (21/2)n^2 + (71/2)n - 7$ \\
    \hline
    $G_{51}$ & $(1/2)n^4 - 2n^3 - (35/2)n^2 + 79n - 46$ & $(1/2)n^4 - (3/2)n^3 - (27/2)n^2 + (87/2)n + 5$ \\
    \hline
    $G_{522}$ & $(1/2)n^4 - 2n^3 - (35/2)n^2 + 89n - 78$ & $(1/2)n^4 - (3/2)n^3 - (27/2)n^2 + (95/2)n - 8$ \\
    \hline
    $G_{517}$ & $(1/2)n^4 - 2n^3 - (35/2)n^2 + 84n - 63$ & $(1/2)n^4 - (3/2)n^3 - (27/2)n^2 + (91/2)n - 2$ \\
    \hline
    $G_{523}$ & $(1/2)n^4 - 2n^3 - (35/2)n^2 + 94n - 96$ & $(1/2)n^4 - (3/2)n^3 - (27/2)n^2 + (99/2)n - 14$ \\
    \hline
  \end{tabular}
\end{table}

\begin{lemma}
Let \(n,k\) be two positive integers with \(n \ge k\) and \(n \ne k-1\). Then
$$\binom{n-1}{k-1} \ge \binom{n-4}{k-1}.$$
\end{lemma}
\begin{proof}
If \(n \le 4\) or \(n < k+3\), then \(\dbinom{n-4}{k-1}=0\), and \(\dbinom{n-1}{k-1}\ge 0\). Hence
\(\binom{n-1}{k-1} \ge \binom{n-4}{k-1}.\)

For \(n \ge k+3\) and \(n>4\), expand via the definition of binomial coefficients:
\[
\binom{n-1}{k-1} = \frac{(n-1)!}{(k-1)!(n-k)!}, \quad \binom{n-4}{k-1} = \frac{(n-4)!}{(k-1)!(n-k-3)!}
\]
take the ratio of the two expressions:
\begin{align*}
\frac{\binom{n-1}{k-1}}{\binom{n-4}{k-1}} =& \frac{(n-1)(n-2)(n-3) \cdot (n-4)!}{(n-k)(n-k-1)(n-k-2)(n-k-3)!} \cdot \frac{(n-k-3)!}{(n-4)!}\\
=&\frac{(n-1)(n-2)(n-3)}{(n-k)(n-k-1)(n-k-2)}
\end{align*}
Since \(n \ge k\) and \(n,k\) are positive integers, we have~$\frac{(n-1)(n-2)(n-3)}{(n-k)(n-k-1)(n-k-2)}\geq1.$~
Therefore,~$\binom{n-1}{k-1} \geq \binom{n-4}{k-1}$~.

\end{proof}
\section{Proof of Theorem 1.1}
From Lemmas 2.4, 2.5 and 2.6, we know that Theorem 1.1 holds for any hook partition \(\lambda=(k,1^{n-k})\) with \(k=1,2\). Therefore, to prove Theorem 1.1, it suffices to treat the case \(k\ge 3\). We outline the main proof idea as follows.According to Table 3, the graph class \(\mathscr{G}_n\) is partitioned into distinct groups by the sum of squares of degrees of its graphs. By Lemmas 2.1 and 2.3, any two graphs from different groups are determined by their hook immanantal polynomials, since they differ either in the number of edges or in the sum of squares of degrees.Furthermore, combining Lemmas 2.7, 2.8, Corollary 2.1 and Tables 1, 2, 3, 4, 5, we compute the 3rd and 4th coefficients of the hook immanantal polynomial for each graph in every group under different hook partitions \(\lambda=(k,1^{n-k})\). We then judge whether these graphs are determined by the k-hook immanantal polynomial via coefficient comparison. Zhang and Wu \cite{H1} presented the numbers of triangles (Table 4) and quadrilaterals (Table 5) for graphs in \(\mathscr{G}_n\), which greatly facilitates the computation of the aforementioned coefficients. Consequently, proving the following lemma is sufficient to establish Theorem 1.1.
\begin{table}[htbp]
\centering
\caption{The number of \(c_3(G)\) in \(\mathscr{G}_n\)}
  \begin{tabular}{|c|l|c|l|}
    \hline
   Graph  & $\mathscr{C}_3(G)$ & Graph & $\mathscr{C}_3(G)$ \\
    \hline
    $G_{10}$ & $\binom{n}{3} - n + 2$ & $G_{521}$ & $\binom{n}{3} - 5n + 10$ \\
    \hline
    $G_{20}$ & $\binom{n}{3} - 2n + 5$ & $G_{32}, G_{33}$ & $\binom{n}{3} - 3n + 8$ \\
    \hline
    $G_{21}$ & $\binom{n}{3} - 2n + 4$ & $G_{41}, G_{47}$ & $\binom{n}{3} - 4n + 11$ \\
    \hline
    $G_{30}$ & $\binom{n}{3} - 3n + 9$ & $G_{59}, G_{524}$ & $\binom{n}{3} - 5n + 17$ \\
    \hline
    $G_{31}$ & $\binom{n}{3} - 3n + 7$ & $G_{42}, G_{45}, G_{48}$ & $\binom{n}{3} - 4n + 10$ \\
    \hline
    $G_{34}$ & $\binom{n}{3} - 3n + 6$ & $G_{44}, G_{46}, G_{410}$ & $\binom{n}{3} - 4n + 12$ \\
    \hline
    $G_{40}$ & $\binom{n}{3} - 4n + 14$ & $G_{50}, G_{51}, G_{514}$ & $\binom{n}{3} - 5n + 12$ \\
    \hline
    $G_{43}$ & $\binom{n}{3} - 4n + 9$ & $G_{55}, G_{512}, G_{520}, G_{523}$ & $\binom{n}{3} - 5n + 15$ \\
    \hline
    $G_{49}$ & $\binom{n}{3} - 4n + 8$ & $G_{58}, G_{511}, G_{513}, G_{517}$ & $\binom{n}{3} - 5n + 13$ \\
    \hline
    $G_{52}$ & $\binom{n}{3} - 5n + 11$ & $G_{53}, G_{57}, G_{518}, G_{519}, G_{522}$ & $\binom{n}{3} - 5n + 14$ \\
    \hline
    $G_{54}$ & $\binom{n}{3} - 5n + 20$ & $G_{56}, G_{510}, G_{515}, G_{516}, G_{525}$ & $\binom{n}{3} - 5n + 16$ \\
    \hline
  \end{tabular}
\end{table}
\begin{lemma}
For any partition \(\lambda=(k,1^{n-k})\) with \(k>2\) and \(n\neq 2k-1\), the complete graph \(K_n\) and the graphs \(G_{10},G_{20},G_{21},G_{31},G_{33},G_{34},G_{40},G_{43},G_{46},G_{49},G_{52},G_{54},G_{521}\) are each uniquely determined by their $k$-th hook immanantal polynomial.
\end{lemma}
\begin{proof}
For the graphs listed above, any two distinct graphs differ either in the number of vertices, the number of edges, or the sum of squares of degrees. It is therefore straightforward to verify that all these graphs are uniquely determined by their $k$-th hook immanantal polynomial.
\end{proof}
\begin{lemma}
Let G be a simple connected graph with n vertices (\(n>7\)) and m edges. For any partition \(\lambda=(k,1^{n-k})\), if \(n\ge k>2\) and \(n\neq 2k-1\), then the graphs \(G_{30}\) and \(G_{32}\) are determined by their $k$-th hook immanantal polynomial.
\end{lemma}
\begin{proof}
From Table 4, Lemma 2.1 and Lemma 2.10, for graphs \(G_{30}\) and \(G_{32}\), when \(k=3\), we have
~$c_{k,3}(G_{30})-c_{k,3}(G_{32})=(F_3(G_{30})-F_3(G_{32}))\binom{n-1}{k-1} + \frac{(2k-n-1)(\mathcal{M}_3^{1}(G_{30})-\mathcal{M}_3^{1}(G_{32}))}{k-1}\binom{n-2}{k-2}-2(|\mathscr{C}_3(G_{30})|
    -|\mathscr{C}_3(G_{32})|)
    \binom{n-4}{k-1}=-2\binom{n-1}{2}+2\binom{n-4}{2}\neq0,$~
When \(k \ge 4\),~$c_{k,3}(G_{30})-c_{k,3}(G_{32})=(F_3(G_{30})-F_3(G_{32}))\binom{n-1}{k-1} + \frac{(2k-n-1)(\mathcal{M}_3^{1}(G_{30})-\mathcal{M}_3^{1}(G_{32}))}{k-1}\binom{n-2}{k-2}-2(|\mathscr{C}_3(G_{30})|
    -|\mathscr{C}_3(G_{32})|)[\binom{n-4}{k-4} - \binom{n-4}{k-1}]=-2\binom{n-1}{k-1}-2[\binom{n-4}{k-4}-\binom{n-4}{k-1}]\leq-2\binom{n-4}{k-4}<0,$~
Therefore, the graphs \(G_{30}\) and \(G_{32}\) are determined by their $k$-th hook immanantal polynomial.
\end{proof}
\begin{table}[htbp]
\centering
\caption{The number of \(c_4(G)\) in \(\mathscr{G}_n\)}
\label{tab:symbol}
  \begin{tabular}{|c|c|c|c|}
    \hline
Graph & $\mathscr{C}_4(G)$ & Graph & $\mathscr{C}_4(G)$ \\
    \hline
    $G_{32}$ & $3\binom{n}{4} - 3n^2 + 18n - 27$ & $G_{510}$ & $3\binom{n}{4} - 5n^2 + 32n - 50$ \\
    \hline
    $G_{33}$ & $3\binom{n}{4} - 3n^2 + 17n - 23$ & $G_{511}$ & $3\binom{n}{4} - 5n^2 + 28n - 26$ \\
    \hline
    $G_{42}$ & $3\binom{n}{4} - 4n^2 + 22n - 22$ & $G_{515}$ & $3\binom{n}{4} - 5n^2 + 31n - 45$ \\
    \hline
    $G_{44}$ & $3\binom{n}{4} - 4n^2 + 24n - 35$ & $G_{516}$ & $3\binom{n}{4} - 5n^2 + 33n - 55$ \\
    \hline
    $G_{46}$ & $3\binom{n}{4} - 4n^2 + 25n - 39$ & $G_{517}$ & $3\binom{n}{4} - 5n^2 + 28n - 27$ \\
    \hline
    $G_{47}$ & $3\binom{n}{4} - 4n^2 + 23n - 29$ & $G_{518}$ & $3\binom{n}{4} - 5n^2 + 30n - 37$ \\
    \hline
    $G_{48}$ & $3\binom{n}{4} - 4n^2 + 22n - 23$ & $G_{520}$ & $3\binom{n}{4} - 5n^2 + 30n - 40$ \\
    \hline
    $G_{410}$ & $3\binom{n}{4} - 4n^2 + 24n - 34$ & $G_{523}$ & $3\binom{n}{4} - 5n^2 + 30n - 39$ \\
    \hline
    $G_{50}$ & $3\binom{n}{4} - 5n^2 + 27n - 21$ & $G_{524}$ & $3\binom{n}{4} - 5n^2 + 32n - 48$ \\
    \hline
    $G_{51}$ & $3\binom{n}{4} - 5n^2 + 27n - 20$ & $G_{41}, G_{45}$ & $3\binom{n}{4} - 4n^2 + 23n - 27$ \\
    \hline
    $G_{53}$ & $3\binom{n}{4} - 5n^2 + 29n - 32$ & $G_{57}, G_{513}$ & $3\binom{n}{4} - 5n^2 + 29n - 30$ \\
    \hline
    $G_{55}$ & $3\binom{n}{4} - 5n^2 + 30n - 38$ & $G_{58}, G_{514}$ & $3\binom{n}{4} - 5n^2 + 28n - 25$ \\
    \hline
    $G_{56}$ & $3\binom{n}{4} - 5n^2 + 31n - 40$ & $G_{512}, G_{525}$ & $3\binom{n}{4} - 5n^2 + 31n - 44$ \\
    \hline
    $G_{59}$ & $3\binom{n}{4} - 5n^2 + 33n - 54$ & $G_{519}, G_{522}$ & $3\binom{n}{4} - 5n^2 + 29n - 33$ \\
    \hline
\end{tabular}
\end{table}
\begin{lemma}
Let G be a simple connected graph with \(n>7\) vertices and m edges. For any partition \(\lambda=(k,1^{n-k})\), if \(k>2\) and \(n\neq 2k-1\), each of the following statements holds:
\begin{enumerate}
 \item[(i)]The graphs \(G_{42}\) and \(G_{48}\) are determined by their $k$-th hook immanantal polynomial.
 \item[(ii)]The graphs \(G_{44}\) and \(G_{410}\) are determined by their $k$-th hook immanantal polynomial.
 \item[(iii)]The graphs \(G_{41}\), \(G_{45}\) and \(G_{47}\) are determined by their $k$-th hook immanantal polynomial.
\end{enumerate}
\end{lemma}
\begin{proof}
\begin{enumerate}
\item[(i)]From Tables 4, 5, Lemma 2.1 and Corollary 2.1, for graphs \(G_{42}\) and \(G_{48}\) with \(n>7\):
When \(k=3\),
$c_{k,4}(G_{42})-c_{k,4}(G_{48})=
    [F_4(G_{42})-F_4(G_{48})]\binom{n-1}{2} + \frac{(5-n)(\mathcal{M}_4^{1}(G_{42})-\mathcal{M}_4^{1}(G_{48}))}{2}\binom{n-2}{1}
    -2(\mathcal{C}_{4}^{3}(G_{42})-\mathcal{C}_{4}^{3}(G_{48}))[ - \binom{n-4}{2}]+2(|\mathscr{C}_4(G_{42})|-|\mathscr{C}_4(G_{48})|)[
    -\binom{n-5}{2}]+[\binom{m}{2}-\sum_{i=1}\limits^n(\binom{d_{42}(v_i)}{2}-
    \binom{d_{48}(v_i)}{2})][-2\binom{n-5}{k-1}]
    =\frac{(5-n)}{2}\binom{n-2}{k-2}-4\binom{n-4}{2}-2\binom{n-5}{2}<0,$
When \(k=4\),
$c_{k,4}(G_{42})-c_{k,4}(G_{48})=\frac{(7-n)}{3}\binom{n-2}{2}+4[1-\binom{n-4}{3}]-2\binom{n-5}{k-1}=\frac{-(n-7)(7n^2-59n+138)}{6},$~
This equation has no positive integer solutions for \(n>7\).
When \(k \ge 5\),
$c_{k,4}(G_{42})-c_{k,4}(G_{48})=
    [F_4(G_{42})-F_4(G_{48})]\binom{n-1}{k-1} + \frac{(2k-n-1)(\mathcal{M}_4^{1}(G_{42})-\mathcal{M}_4^{1}(G_{48}))}{k-1}\binom{n-2}{k-2}
    -2(\mathcal{C}_{4}^{3}(G_{42})-\mathcal{C}_{4}^{3}(G_{48}))[\binom{n-4}{k-4} - \binom{n-4}{k-1}]+2(|\mathscr{C}_4(G_{42})|-|\mathscr{C}_4(G_{48})|)[\binom{n-5}{k-5}
    -\binom{n-5}{k-1}]+[\binom{m}{2}-\sum_{i=1}\limits^n(\binom{d_{42}(v_i)}{2}-
    \binom{d_{48}(v_i)}{2})][\binom{n-5}{k-5}-2\binom{n-5}{k-3}+\binom{n-5}{k-1}]
    =\frac{(2k-n-1)}{k-1}\binom{n-2}{k-2}+4[\binom{n-4}{k-4} - \binom{n-4}{k-1}]+2[\binom{n-5}{k-5}-\binom{n-5}{k-1}]\neq0,$
Therefore, the graphs \(G_{42}\) and \(G_{48}\) are determined by their $k$-th hook immanantal polynomial.
\item[(ii)]
From Table 4 and Lemma 2.1, for graphs \(G_{44}\) and \(G_{410}\),
$c_{k,3}(G_{44})-c_{k,3}(G_{410})=(F_3(G_{44})-F_3(G_{410}))\binom{n-1}{k-1} + \frac{(2k-n-1)(\mathcal{M}_3^{1}(G_{44})-\mathcal{M}_3^{1}(G_{410}))}{k-1}\binom{n-2}{k-2}-2(|\mathscr{C}_3(G_{44})|
    -|\mathscr{C}_3(G_{410})|)(\binom{n-4}{k-4} - \binom{n-4}{k-1})=2\binom{n-1}{k-1}\neq0$,
Therefore, the graphs \(G_{44}\) and \(G_{410}\) are determined by their $k$-th hook immanantal polynomial.
\item[(iii)]From Table 4, Lemma 2.1 and Lemma 2.10, for graphs \(G_{41}\) and \(G_{45}\):
When \(k=3\),
 $c_{k,3}(G_{41})-c_{k,3}(G_{45})=-2\binom{n-1}{2}+2\binom{n-4}{2}\neq0,$~
 When \(k \ge 4\),
 $c_{k,3}(G_{41})-c_{k,3}(G_{45})=-2\binom{n-1}{k-1}-2[\binom{n-4}{k-4}-\binom{n-4}{k-1}]\leq-2\binom{n-4}{k-4}<0,$~
 Thus, the graphs \(G_{41}\) and \(G_{45}\) are determined by their $k$-th hook immanantal polynomial.
 Meanwhile, for graphs \(G_{41}\) and \(G_{47}\),
 ~$c_{k,3}(G_{41})-c_{k,3}(G_{47})=-2\binom{n-1}{k-1}\neq0,$~
\(G_{41}\) and \(G_{47}\) are determined by their $k$-th hook immanantal polynomial.
Similarly, for graphs \(G_{45}\) and \(G_{47}\),
~$c_{k,3}(G_{45})-c_{k,3}(G_{47})=2[\binom{n-4}{k-4}-\binom{n-4}{k-1}],$~
Since \(n \neq 2k-1\), we have~$2[\binom{n-4}{k-4}-\binom{n-4}{k-1}]\neq0,$~
Therefore, the graphs \(G_{45}\) and \(G_{47}\) are determined by their $k$-th hook immanantal polynomial.
\end{enumerate}
\end{proof}
\begin{lemma}
Let G be a simple connected graph with \(n>7\) vertices and m edges. For any partition \(\lambda=(k,1^{n-k})\), if \(k>2\) and \(n\neq 2k-1\), each of the following statements holds:
\begin{enumerate}
\item[(i)] The graphs \(G_{56}, G_{512}, G_{515}\) and \(G_{525}\) are determined by their $k$-th hook immanantal polynomial.
\item[(ii)] The graphs \(G_{58}, G_{511}, G_{514}\) and \(G_{517}\) are determined by their $k$-th hook immanantal polynomial.
\item[(iii)] The graphs \(G_{510}\) and \(G_{524}\) are determined by their $k$-th hook immanantal polynomial.
\item[(iv)]The graphs \(G_{53}, G_{57}, G_{513}, G_{519}\) and \(G_{522}\) are determined by their $k$-th hook immanantal polynomial.
\item[(v)] The graphs \(G_{55}, G_{518}, G_{520}\) and \(G_{523}\) are determined by their $k$-th hook immanantal polynomial.
\item[(vi)] The graphs \(G_{59}\) and \(G_{516}\) are determined by their $k$-th hook immanantal polynomial.
\item[(vii)] The graphs \(G_{50}\) and \(G_{51}\) are determined by their $k$-th hook immanantal polynomial.
\end{enumerate}
\end{lemma}
\begin{proof}
\begin{enumerate}
\item[(i)]
From Table 4, Lemma 2.1 and Lemma 2.10, for graphs \(G_{56}\) and \(G_{512}\):When \(k=3\),
~$c_{k,3}(G_{56})-c_{k,3}(G_{512})=-6\binom{n-1}{2}+2\binom{n-4}{2}\leq-4\binom{n-1}{2}<0,$~
When \(k \ge 4\),
~$c_{k,3}(G_{56})-c_{k,3}(G_{512})=-6\binom{n-1}{k-1}-2[\binom{n-4}{k-4}-\binom{n-4}{k-1}]\leq-
    4\binom{n-1}{k-1}-2\binom{n-4}{k-4}<0,$~
Thus, \(G_{56}\) and \(G_{512}\) are determined by their~$k$-th hook immanantal polynomial.
Similarly,
~$c_{k,3}(G_{56})-c_{k,3}(G_{515})=-6\binom{n-1}{k-1}\neq0,$~
    ~$c_{k,3}(G_{56})-c_{k,3}(G_{525})=-4\binom{n-1}{k-1}\neq0,$~
so \(G_{56}\) is  determined from both \(G_{515}\) and \(G_{525}\) by the~$k$-th hook immanantal polynomial. For the previously missing pair \(G_{512},G_{515}\), the first differing coefficient is
	\[
	c_{k,3}(G_{512})-c_{k,3}(G_{515})
	=\begin{cases}
		-2\binom{n-4}{2}<0,& k=3,\\
		2\left[\binom{n-4}{k-4}-\binom{n-4}{k-1}\right]\neq0,& k\geq4,
	\end{cases}
	\]
	where the second line is nonzero because \(n\neq2k-1\). For the pair
	\(G_{512},G_{525}\), the difference is
	\[
	c_{k,3}(G_{512})-c_{k,3}(G_{525})
	=\begin{cases}
		2\binom{n-1}{2}-2\binom{n-4}{2}\neq0,& k=3,\\
		2\binom{n-1}{k-1}
		+2\left[\binom{n-4}{k-4}-\binom{n-4}{k-1}\right]>0,& k\geq4,
	\end{cases}
	\]
	because \(\binom{n-1}{k-1}\geq\binom{n-4}{k-1}\). Thus all six unordered
	pairs in this four-graph group have a nonzero coefficient certificate.
For graphs \(G_{515}\) and \(G_{525}\),
  $c_{k,3}(G_{515})-c_{k,3}(G_{525})=2\binom{n-1}{k-1}\neq0,$~
so \(G_{515}\) and \(G_{525}\) are determined by their $k$-th hook immanantal polynomial.
\item[(ii)]
From Tables 4, 5, Lemma 2.1, Corollary 2.1 and Lemma 2.10:
For graphs \(G_{58}\) and \(G_{511}\),
~$c_{k,3}(G_{58})-c_{k,3}(G_{511})=-2\binom{n-1}{k-1}\neq0.$~
For graphs \(G_{58}\) and \(G_{514}\):
When \(k=3\),
 $c_{k,3}(G_{58})-c_{k,3}(G_{514})=2\binom{n-1}{2}-2\binom{n-4}{2}\neq0,$~
 When \(k \ge 4\),
   ~$c_{k,3}(G_{58})-c_{k,3}(G_{514})=-2\binom{n-1}{k-1}-2[\binom{n-4}{k-4}-\binom{n-4}{k-1}]\leq-2\binom{n-4}{k-4}<0.$~
 For graphs \(G_{58}\) and \(G_{517}\),
  ~$c_{k,3}(G_{58})-c_{k,3}(G_{517})=-2\binom{n-1}{k-1}\neq0.$~
  For graphs \(G_{511}\) and \(G_{514}\):
  When \(k=3\),
   $c_{k,3}(G_{511})-c_{k,3}(G_{514})=2\binom{n-4}{2}>0,$
   When \(k \ge 4\),
   ~$c_{k,3}(G_{511})-c_{k,3}(G_{514})=-2[\binom{n-4}{k-4}-\binom{n-4}{k-1}],$~
   Since \(n \neq 2k-1\), we have
     ~$-2[\binom{n-4}{k-4}-\binom{n-4}{k-1}]\neq0.$~
   For graphs \(G_{511}\) and \(G_{517}\) with \(n>7\):
   When \(k=3\),
   ~$c_{k,4}(G_{511})-c_{k,4}(G_{517})=\frac{(5-n)}{2}\binom{n-2}{1}-4 \binom{n-4}{2}-2\binom{n-5}{2}<0,$~
   When \(k=4\),
    ~$c_{k,4}(G_{511})-c_{k,4}(G_{517})=\frac{(7-n)}{3}\binom{n-2}{2}+4[1-\binom{n-4}{3}]-2\binom{n-5}{k-1}=\frac{-(n-7)(7n^2-59n+138)}{6},$~ which has no positive integer solutions for \(n>7\).
When \(k \ge 5\),
 ~$c_{k,4}(G_{511})-c_{k,4}(G_{517})=\frac{(2k-n-1)}{k-1}\binom{n-2}{k-2}+4[\binom{n-4}{k-4}- \binom{n-4}{k-1}]+2[\binom{n-5}{k-5}-\binom{n-5}{k-1}]\neq0.$~
 For graphs \(G_{514}\) and \(G_{517}\):
 When \(k=3\),
  ~$c_{k,3}(G_{514})-c_{k,3}(G_{517})=-2\binom{n-4}{2}<0,$~
  When \(k \ge 4\),
   ~$c_{k,3}(G_{514})-c_{k,3}(G_{517})=2[\binom{n-4}{k-4}-\binom{n-4}{k-1}],$~
   Since \(n \neq 2k-1\), we have  ~$2[\binom{n-4}{k-4}-\binom{n-4}{k-1}]\neq0.$~
   Thus every graph in the set \(\{G_{58}, G_{511}, G_{514}, G_{517}\}\) is determined by its $k$-th hook immanantal polynomial.
\item[(iii)]
From Table 4, Lemma 2.1 and Lemma 2.10, for graphs \(G_{510}\) and \(G_{524}\):
When \(k=3\),
 ~$c_{k,3}(G_{510})-c_{k,3}(G_{524})=4\binom{n-1}{2}-2\binom{n-4}{2}\geq2\binom{n-1}{2}>0,$~
 When \(k \ge 4\),
 ~$c_{k,3}(G_{510})-c_{k,3}(G_{524})=4\binom{n-1}{k-1}+2[\binom{n-4}{k-4}-\binom{n-4}{k-1}]\geq2\binom{n-1}{k-1}+2\binom{n-4}{k-4}>0.$~
 Therefore, the graphs \(G_{510}\) and \(G_{524}\) are determined by their $k$-th hook immanantal polynomial.
\item[(iv)]
  From Tables 4, 5, Lemma 2.1, Corollary 2.1 and Lemma 2.10:
For graphs \(G_{53}\) and \(G_{57}\) with \(n>7\):
When \(k=3\),
~$c_{k,4}(G_{53})-c_{k,4}(G_{57})=-(5-n)\binom{n-2}{1}+8\binom{n-4}{2}+4\binom{n-5}{2}>0,$~
When \(k=4\),
 ~$c_{k,4}(G_{53})-c_{k,4}(G_{57})=-\frac{2(7-n)}{3}\binom{n-2}{2}-8[1-\binom{n-4}{3}]+4\binom{n-5}{3}=\frac{-(n-7)(7n^2-59n+138)}{3},$~
which has no integer solutions making the difference zero for all \(n>7\).
When \(k \ge 5\),
 ~$c_{k,4}(G_{53})-c_{k,4}(G_{57})=-2\frac{(2k-n-1)}{k-1}\binom{n-2}{k-2}-8[\binom{n-4}{k-4}- \binom{n-4}{k-1}]-4[\binom{n-5}{k-5}-\binom{n-5}{k-1}]\neq0,$~
Thus \(G_{53}\) and \(G_{57}\) are determined by their $k$-th hook immanantal polynomial.
For graphs \(G_{53}\) and \(G_{513}\):
When \(k=3\),
 ~$c_{k,3}(G_{53})-c_{k,3}(G_{513})=-2\binom{n-1}{2}+2\binom{n-4}{2}\neq0,$~
When \(k \ge 4\),
 ~$c_{k,3}(G_{53})-c_{k,3}(G_{513})=-2\binom{n-1}{k-1}-2[\binom{n-4}{k-4}-\binom{n-4}{k-1}]\leq-2\binom{n-4}{k-4}<0,$~
Hence \(G_{53}\) and \(G_{513}\) are determined by their $k$-th hook immanantal polynomial.
For graphs \(G_{53}\) and \(G_{519}\),
 ~$c_{k,3}(G_{53})-c_{k,3}(G_{519})=-2\binom{n-1}{k-1}\neq0,$~
so \(G_{53}\) and \(G_{519}\) are determined by their \(k\)-th hook immanantal polynomial.
For graphs \(G_{53}\) and \(G_{522}\),
 ~$c_{k,3}(G_{53})-c_{k,3}(G_{522})=-2\binom{n-1}{k-1}\neq0,$~
so \(G_{53}\) and \(G_{522}\) are determined by their \(k\)-th hook immanantal polynomial.
For graphs \(G_{57}\) and \(G_{513}\):
When \(k=3\),
  $c_{k,3}(G_{57})-c_{k,3}(G_{513})=-2\binom{n-1}{2}+2\binom{n-4}{2}\neq0,$~
When \(k \ge 4\),
~$c_{k,3}(G_{57})-c_{k,3}(G_{513})=-2\binom{n-1}{k-1}-2[\binom{n-4}{k-4}-\binom{n-4}{k-1}]\leq-2\binom{n-4}{k-4}<0,$~
Thus \(G_{57}\) and \(G_{513}\) are determined by their \(k\)-th hook immanantal polynomial.
For graphs \(G_{57}\) and \(G_{519}\),
$c_{k,3}(G_{57})-c_{k,3}(G_{519})=-2\binom{n-1}{k-1}<0,$~
so \(G_{57}\) and \(G_{519}\) are determined by their \(k\)-th hook immanantal polynomial.
For graphs \(G_{57}\) and \(G_{522}\),
 ~$c_{k,3}(G_{57})-c_{k,3}(G_{522})=-2\binom{n-1}{k-1}\neq0,$~
so \(G_{57}\) and \(G_{522}\) are determined by their \(k\)-th hook immanantal polynomial.
 For graphs \(G_{513}\) and \(G_{519}\):
When \(k=3\),
$c_{k,3}(G_{513})-c_{k,3}(G_{519})=-2\binom{n-4}{2}<0,$~
When \(k \ge 4\),
 $c_{k,3}(G_{513})-c_{k,3}(G_{519})=2[\binom{n-4}{k-4}-\binom{n-4}{k-1}],$~
Since \(n \neq 2k-1\), we obtain
~$2[\binom{n-4}{k-4}-\binom{n-4}{k-1}]\neq0,$~
so the graphs \(G_{513}\) and \(G_{519}\) are determined by their \(k\)-th hook immanantal polynomial.
For graphs \(G_{513}\) and \(G_{522}\):
When \(k=3\),
 ~$c_{k,3}(G_{513})-c_{k,3}(G_{522})=-2\binom{n-4}{2}<0,$~
When \(k \ge 4\),
 ~$c_{k,3}(G_{513})-c_{k,3}(G_{522})=2[\binom{n-4}{k-4}-\binom{n-4}{k-1}],$~
As \(n \neq 2k-1\), we have
~$2[\binom{n-4}{k-4}-\binom{n-4}{k-1}]\neq0,$~
thus the graphs \(G_{513}\) and \(G_{522}\) are determined by their \(k\)-th hook immanantal polynomial.
For graphs \(G_{519}\) and \(G_{522}\) with \(n>7\):
When \(k=3\),
 ~$c_{k,4}(G_{519})-c_{k,4}(G_{522})=-\frac{(5-n)}{2}\binom{n-2}{1}+4\binom{n-5}{2}>0,$~
When \(k=4\),
 ~$c_{k,4}(G_{519})-c_{k,4}(G_{522})=-\frac{(7-n)}{3}\binom{n-2}{2}+4\binom{n-5}{3}=\frac{(n-7)(5n^2-49n+126)}{6},$~
If this difference equals zero, the above equation admits no positive integer solutions satisfying \(n>7\).
When \(k \ge 5\),
  ~$c_{k,4}(G_{519})-c_{k,4}(G_{522})=-\frac{(2k-n-1)}{k-1}\binom{n-2}{k-2}-4[\binom{n-5}{k-5}-\binom{n-5}{k-1}]\neq0,$~
Therefore, the graphs \(G_{519}\) and \(G_{522}\) are determined by their \(k\)-th hook immanantal polynomial.
 \item[(v)]
 From Table 4, Lemma 2.1, Corollary 2.1, Lemma 2.19 and Lemma 2.10:For graphs \(G_{55}\) and \(G_{518}\):
When \(k=3\),
 ~$c_{k,3}(G_{55})-c_{k,3}(G_{518})=2\binom{n-4}{2}>0,$~
 When \(k \ge 4\),
  ~$c_{k,3}(G_{55})-c_{k,3}(G_{518})=-2[\binom{n-4}{k-4}-\binom{n-4}{k-1}],$~
  Since \(n \neq 2k-1\), we have~$-2[\binom{n-4}{k-4}-\binom{n-4}{k-1}]\neq0,$~
  Thus \(G_{55}\) and \(G_{518}\) are determined by their $k$-th hook immanantal polynomial.
  For graphs \(G_{55}\) and \(G_{520}\),
   ~$c_{k,3}(G_{55})-c_{k,3}(G_{520})=-2\binom{n-1}{k-1}\neq0,$~
   so \(G_{55}\) and \(G_{520}\) are determined by their $k$-th hook immanantal polynomial.
   For graphs \(G_{55}\) and \(G_{523}\) with \(n>7\):
   When \(k=3\),
   $c_{k,4}(G_{55})-c_{k,4}(G_{523})=\frac{(5-n)}{2}\binom{n-2}{1}-4\binom{n-4}{2}-2\binom{n-5}{2}<0,$~
   When \(k=4\),
    ~$c_{k,4}(G_{55})-c_{k,4}(G_{523})=\frac{(7-n)}{3}\binom{n-2}{2}+4[1-\binom{n-4}{3}]-2\binom{n-5}{3}=\frac{-(n-7)(7n^2-59n+138)}{6},$~
    which has no positive integer solutions for \(n>7\).
    When \(k \ge 5\),
     ~$c_{k,4}(G_{55})-c_{k,4}(G_{523})=\frac{(2k-n-1)}{k-1}\binom{n-2}{k-2}+4[\binom{n-4}{k-4} - \binom{n-4}{k-1}]+2[\binom{n-5}{k-5}-\binom{n-5}{k-1}]\neq0,$~
     Hence \(G_{55}\) and \(G_{523}\) are determined by their $k$-th hook immanantal polynomial.
     For graphs \(G_{518}\) and \(G_{520}\):
     When \(k=3\),
       ~$c_{k,3}(G_{518})-c_{k,3}(G_{520})=-2\binom{n-1}{2}-2\binom{n-4}{2}<0,$~
     When \(4 \le k \le n-3\),
       ~$c_{k,3}(G_{518})-c_{k,3}(G_{520})=-2\binom{n-1}{k-1}+2[\binom{n-4}{k-4}-\binom{n-4}{k-1}]\leq-2\binom{n-4}{k-1}<0,$~
       When \(k = n-2\),
       $c_{k,3}(G_{518})-c_{k,3}(G_{520})=-2\binom{n-1}{n-3}+2\binom{n-4}{n-6}\neq0,$
       When \(k = n-1\),
       $c_{k,3}(G_{518})-c_{k,3}(G_{520})=-2\binom{n-1}{n-2}+2\binom{n-4}{n-5}\neq0,$
       When \(k = n\),
       ~$c_{k,3}(G_{518})-c_{k,3}(G_{520})=0,~c_{k,4}(G_{518})-c_{k,4}(G_{520})=-1,$~
       Therefore \(G_{518}\) and \(G_{520}\) are determined by their $k$-th hook immanantal polynomial.
     For graphs \(G_{518}\) and \(G_{523}\):
     When \(k=3\),
      ~$c_{k,3}(G_{518})-c_{k,3}(G_{523})=-2\binom{n-4}{2}<0,$~
      When \(k \ge 4\),
      ~$c_{k,3}(G_{518})-c_{k,3}(G_{523})=2[\binom{n-4}{k-4}-\binom{n-4}{k-1}],$~
      Since \(n \neq 2k-1\), we have \(2[\binom{n-4}{k-4}-\binom{n-4}{k-1}] \neq 0\).
      Thus \(G_{518}\) and \(G_{523}\) are determined by their $k$-th hook immanantal polynomial.
      For graphs \(G_{520}\) and \(G_{523}\),
       ~$c_{k,3}(G_{520})-c_{k,3}(G_{523})=2\binom{n-1}{k-1}\neq0,$~
       so \(G_{520}\) and \(G_{523}\) are determined by their $k$-th hook immanantal polynomial.
   \item[(vi)]
   From Table 4, Lemma 2.1 and Lemma 2.10, for graphs \(G_{59}\) and \(G_{516}\):
   When \(k=3\),
    ~$c_{k,3}(G_{59})-c_{k,3}(G_{516})=-4\binom{n-1}{2}-2\binom{n-4}{2}<0,$~
    When \(k \ge 4\),
    ~$c_{k,3}(G_{59})-c_{k,3}(G_{516})=-4\binom{n-1}{k-1}-2[\binom{n-4}{k-4}-\binom{n-4}{k-1}]\leq-2\binom{n-1}{k-1}-2\binom{n-4}{k-4}<0,$~
   Thus, the graphs \(G_{59}\) and \(G_{516}\) are determined by their $k$-th hook immanantal polynomial.
   \item[(vii)]
 From Table 4, Lemma 2.1 and Corollary 2.1, for graphs \(G_{50}\) and \(G_{51}\) with \(n>7\):
 When \(k=3\),
~$c_{k,4}(G_{50})-c_{k,4}(G_{51})=-\frac{(5-n)}{2}\binom{n-2}{1}+4 \binom{n-4}{2}]+2\binom{n-5}{2}>0,$~
When \(k=4\),
 ~$c_{k,4}(G_{50})-c_{k,4}(G_{51})=-\frac{(7-n)}{3}\binom{n-2}{2}-4[1-\binom{n-4}{3}]+2\binom{n-5}{3}=\frac{(n-7)(7n^2-59n+138)}{6},$~
 which has no positive integer solutions for \(n>7\).
 When \(k \ge 5\),
 ~$c_{k,4}(G_{50})-c_{k,4}(G_{51})=-1\frac{(2k-n-1)}{k-1}\binom{n-2}{k-2}-4[\binom{n-4}{k-4} - \binom{n-4}{k-1}]-2[\binom{n-5}{k-5}-\binom{n-5}{k-1}]\neq0,$~
 Therefore, the graphs \(G_{50}\) and \(G_{51}\) are determined by their $k$-th hook immanantal polynomial.
\end{enumerate}
\end{proof}
\begin{proof}[Proof of Theorem~\ref{thm:main}]
Theorem ~\ref{thm:main} can be established directly from Lemmas 2.3-2.6 and Lemmas 3.1--3.4.
\end{proof}

\section{ Conclusion}
In this paper, we mainly prove that for any fixed partition \((k,1^{n-k})\) with \(n \neq 2k-1\), every subgraph obtained by deleting at most five edges from the complete graph is determined by its $k$-th hook immanantal polynomial of the Laplacian matrix.In fact, from Lemma 2.1 we know that when \(n=2k-1\), the binomial differences
\(\binom{n-4}{k-4}-\binom{n-4}{k-1},\quad \binom{n-5}{k-5}-\binom{n-5}{k-1}\)
equal zero. Hence the method used to prove Theorem 1.1 fails in this case, and new techniques need to be developed. Furthermore, many interesting problems remain for further investigation:
\begin{enumerate}
	
	\item[(i)] Are the graphs in \(\mathscr{G}_n\) determined by the \(k\)-th hook immanantal polynomial when \(n=2k-1\)?
	
\item[(ii)] What families of graphs are determined by the $k$-th hook immanantal polynomial of their graph matrix?
\end{enumerate}

\end{document}